\documentclass[a4paper,11pt,reqno]{amsart}
\usepackage{graphicx, color} 
\usepackage{amsmath,amsthm,amsfonts,amssymb,amscd}
\usepackage[T1]{fontenc}

\usepackage{hyperref} 

\theoremstyle{plain}
\newtheorem{Teo}{Theorem}[section]

\newtheorem{Ex}[Teo]{Example}
\newtheorem{Lema}[Teo]{Lemma}
\newtheorem{Prop}[Teo]{Proposition}
\newtheorem{Cor}[Teo]{Corollary}
\newtheorem{maintheorem}{Theorem}
\newtheorem{maincorollary}[maintheorem]{Corollary}

\theoremstyle{remark}
\newtheorem{Remark}[Teo]{Remark}
\newcommand{\ds}{\displaystyle}

\newcommand{\inter}{\operatorname{int}}
\newcommand{\diam}{\operatorname{diam}}
\newcommand{\supp}{\operatorname{supp}}
\renewcommand{\top}{\operatorname{top}}

\begin{document}

\title[Equilibrium stability]{Equilibrium stability for non-uniformly hyperbolic systems}

\author[J. F. Alves]{J. F. Alves}
\address{Jos\'{e} F. Alves\\ Centro de Matem\'{a}tica da Universidade do Porto\\ Rua do Campo Alegre 687\\ 4169-007 Porto\\ Portugal}
\email{jfalves@fc.up.pt} \urladdr{http://www.fc.up.pt/cmup/jfalves}

\author[V. Ramos]{V. Ramos }
\address{Vanessa Ramos \\ Centro de Ci\^{e}ncias Exatas e Tecnologia-UFMA\\ Av. dos Portugueses, 1966, Bacanga\\  65080-805 S\~{a}o Lu\'{i}s\\Brasil}
\email{vramos@impa.br}

\author[J. Siqueira]{J. Siqueira}
\address{Jaqueline Siqueira\\ Departamento de Matem\'{a}tica PUC-Rio, Marqu\^{e}s de S\^{a}o Vicente 225, G\'{a}vea\\ 225453-900 Rio de Janeiro\\Brazil}
\email{jaqueline@mat.puc-rio.br}

\date{}
\thanks{JFA was partially supported by Funda\c c\~ao Calouste Gulbenkian and by   CMUP (UID/MAT/00144/2013), PTDC/MAT-CAL/3884/2014 and FAPESP/19805/2014 which are funded by FCT (Portugal) with national (MEC) and European structural funds through the programs COMPTE and FEDER, under the partnership agreement PT2020. JS was supported by CNPq-Brazil. VR was supported by CNPq-Brazil and by FAPEMA-Brazil.}

\pagenumbering{arabic}

\begin{abstract}We prove that for a wide family of non-uniformly hyperbolic maps and hyperbolic potentials we have equilibrium stability, i.e. the equilibrium states depend continuously on the dynamics and the potential. For this we deduce that the topological pressure is continuous as a function of the dynamics and the potential.
We  also prove the existence of finitely many ergodic equilibrium states for non-uniformly hyperbolic skew products and hyperbolic H\"older continuous potentials. Finally we show that these equilibrium states vary continuously in the weak$^\ast$ topology within such systems.
\end{abstract}

\maketitle



\tableofcontents


\section{Introduction}
\label{introducao}

The study of probability measures which remain invariant under the action of a dynamical system provides relevant information about the topological behavior of its orbits. For instance, Poincar\'e's Recurrence Theorem states that the orbit of almost every point with respect  to any invariant probability measure returns arbitrarily close to its initial state. When the dynamics has more than one invariant probability measure, an efficient way to choose an interesting one is to select those maximizing the free energy of the system, which are called equilibrium states.

More formally, given a continuous map $T: X \to X$ defined on a compact metric space $X$ and a  continuous potential $\phi: X \to \mathbb{R}$, we say that a $T$-invariant probability measure $\mu$ on the Borel sets of $X$ is an {\em equilibrium state} for $(T,\phi)$  if it satisfies the following variational principle:
          $$h_{\mu} (T) + \int \phi \, d\mu  = \sup_{\eta\in\mathbb P_T(X)} \left\{ h_{\eta} (T) + \int \phi \, d\eta 
\right\},  $$
 where $\mathbb P_T(M)$ denotes the set of $T$-invariant probability measures on the Borel sets of $X$,  that we always consider   endowed with the weak* topology.

The theory of equilibrium states was initiated by Sinai, Ruelle and Bowen in the seventies through the application of techniques and results from statistical mechanics to smooth dynamical systems. In the pioneering work \cite{Sinai} Sinai studied the problem of existence and finiteness of equilibrium state for Anosov diffeomorphisms and  H\"older continuous potentials. This strategy was carried out by Ruelle and Bowen in \cite{Ruelle68}, \cite{Ruelle78} and \cite{Bowen71} to extend the theory to uniformly hyperbolic (Axiom A) dynamical systems.

In the non-uniformly hyperbolic setting in dimension greater than one several advances were obtained by Sarig (see \cite{Sarig99}, \cite{Sarig03}) and Buzzi   \cite{BuzziSarig}, who studied countable Markov shifts, and Buzzi, Paccaut and Schmitt \cite{Buzzi},  who studied piecewise expanding maps.   Arbieto, Matheus and Oliveira \cite{AMO}, Oliveira and Viana \cite{OV08}, and Varandas and Viana \cite{VV} studied certain classes of non-uniformly expanding maps.

The problem of existence and finiteness of equilibrium states for partially hyperbolic systems has been fraught with greater challenges. In \cite{Buzziecia}, Buzzi, Fisher, Sambarino and V\'asquez obtained uniqueness of the maximal entropy measure for partially hyperbolic maps derived from Anosov. Climenhaga, Fisher and Thompson in \cite{CFT16, CFT17} address the question of existence and uniqueness of equilibrium states for Bonatti-Viana diffeomorphisms and Ma\~n\'e diffeomorphisms for suitable classes of potentials.
Castro and Nascimento in \cite{CN} showed uniqueness of the maximal entropy measure
for partially hyperbolic attractors semiconjugated to nonuniformly expanding maps. For a family of partially hyperbolic horseshoes introduced by D\'\i az, Horita, Rios and Sambarino in \cite{diazetal} the existence of equilibrium states for any continuous potential was proved by Leplaideur, Oliveira and Rios in \cite{LOR}. 
Later, Rios and Siqueira \cite{RS15} proved uniqueness of equilibrium states for a class of H\"older continuous potentials with small variation and which do not depend on the the stable direction. Recently, Ramos and Siqueira \cite{RS16} extended  this result to a broader class of H\"older continuous potentials.

Once we have established the existence and  finiteness of equilibrium states a natural question that arises is {\em how does the equilibrium states vary with the underlying dynamics and the potential.} Actually, one of the main goals in the field of Dynamical Systems is to understand how the behavior of the system changes under perturbations of  the underlying dynamics. The concept of \emph{structural stability} was introduced by Andronov and Potryagin \cite{APo} and states that the whole orbit structure remain unchanged under small perturbations. Nevertheless structural stability has been proved to be a concept too strong in the sense that many relevant models are not structural stable while     some of their dynamics properties remain unchanged after small perturbation.  Motivated by this, Alves and Viana introduced in \cite{AV}  the notion of \emph{statistical stability} which expresses the persistence of statistical properties in terms of the continuity of the physical measure. 
In the aforementioned work Alves and Viana considered a robust class of maps with non-uniform expansion and proved that the referent (unique) SRB measure varies continuously with the dynamics in the $L^1$-norm.

In the present work we introduce the concept of \emph{equilibrium stability} phrasing  the persistence of the equilibrium states under small perturbations of the underlying dynamics and the potentials. 
In  \cite{VV} Varandas and Viana proved statistical stability for a class of non-uniformly expanding local homeomorphisms on compact manifolds. Moreover, they show that if the topological pressure function depends continuously within such systems, then equilibrium stability holds for H\"older continuous potentials with not very large variation.  
The subsequent works \cite{BCV,CV} by Bomfim, Castro and Varandas study the continuity and even differentiability of several thermodynamical quantities for certain classes of non-uniformly expanding dynamical systems and potentials with small variation.
Here  we consider a family of  non-uniformly expanding maps and hyperbolic potentials and  prove that the topological pressure varies continuously  within this family. This enables us to derive that this family is equilibrium stable.

We also consider skew products over non-uniformly expanding maps with uniform contraction on the fiber.
For this kind of  systems Ramos and Viana in \cite{RV16} showed the existence of finitely many ergodic equilibrium states for hyperbolic potentials which do not depend on the stable direction. Here we give a sufficient condition on the fiber dynamic to enlarge the class of potentials for which there exists finiteness. 
We use the strategy of Ramos and Siqueira \cite{RS16} to prove that the independence of the stable direction required in \cite{RV16} is not necessary. We prove that given a hyperbolic H\"older continuous potential we can construct a homologous potential which is still hyperbolic and H\"older continuous but does not depend on the stable direction. This allows us to prove finiteness of equilibrium states.  Finally we obtain  equilibrium stability for a family of skew products and hyperbolic potentials.   We believe that these results can be an important step towards understanding  the more difficult class of partially hyperbolic diffeomorphisms with a non-uniformly expanding central direction and a uniformly contracting direction.

%
%
%
%
%
%

This work is organized as follows. In Section~\ref{Results} we define the equilibrium stability and precisely state our main results. In Section~\ref{hyp pot} we recall the definition of relative pressure and some classical results in functional analysis. In Section~\ref{stability} we show the continuity of the topological pressure as a function of the base dynamics and the potential. From this result we derive the stability of equilibrium states associated to non-uniformly expanding maps and hyperbolic potentials. In Section~\ref{est eq} we prove the existence of finitely many ergodic equilibrium states and we extend the equilibrium stability obtained for the base dynamics to the skew product. We also prove the continuity of the topological pressure function. Finally, in Section~\ref{applications} we describe some classes of examples which satisfy our results.


\section{Statement of results}\label{Results} 
Let $M$ be a compact Riemannian manifold and   $\mathcal F$ a family of $C^1$ local diffeomorphisms $f:M\to M$.  Given  $\alpha>0$, consider $  C^\alpha(M)$ the space of H\"older continuous potentials $\phi:M\to\mathbb R$  endowed with the seminorm
 $$|\varphi|_\alpha=\sup_{x\neq y}\frac{|\varphi(x)-\varphi(y)|}{d(x,y)^\alpha}
 $$
and the norm
 $$\|\varphi\|_\alpha=\|\varphi\|_0+|\varphi|_\alpha,$$
 where $\|\quad\|_0$ stands for the $\sup$ norm in $C^0(M)$.
%
We shall always consider  $  \mathcal F\times  C^\alpha(M)$ endowed  with the product topology.  Assume that $\mathcal H$ is a subset  of $  \mathcal F\times  C^\alpha(M)$    such  that  each $(f,\phi)\in \mathcal H$ has a unique equilibrium state~$\mu_{f,\phi}$. We say that  $ \mathcal H$ is  {\em equilibrium stable}  if the function assigning to each 
$(f,\phi) \in \mathcal H$ its unique equilibrium state $ \mu_{f,\phi}\in \mathbb P_f(M)$
is continuous.
%
%

\subsection{Non-uniform expansion}\label{se.NUE}
Given $c>0$, define $\Sigma_c(f)$ as the set of points $x\in M$ where $f$ is \emph{non-uniformly expanding}, i.e. 
\begin{eqnarray} \label{limsup}
\limsup_{n\rightarrow+\infty}\frac{1}{n}\sum_{i=0}^{n-1}\log\|Df(f^{j}(x))^{-1}\|\leq-c.
\end{eqnarray}
We say that   $\phi: M \to \mathbb{R}$ continuous is a {\em $c$-hyperbolic potential for~$f$} if  the  topological pressure of $\phi$ (with respect to $f$) is equal to the relative pressure of $\phi$ on the set~$\Sigma_c(f)$; for the definition of topological pressure relative to a set see Section~\ref{hyp pot}.  Proposition~\ref{pr.openess}  gives in particular that the set of $c$-hyperbolic potentials for $f$ is an open subset of $C^0(M)$. The existence of only finitely many ergodic equilibrium states for $(f, \phi) $ with $c$-hyperbolic $\phi$ was established in~\cite[Theorem~2]{RV16}. Additionally, under the  assumption that
\begin{equation}\tag{$\ast$}\label{propestrela}
\text{$\{f^{-n}(x)\}_{n\ge 0}$ is dense in $M$ for all $x\in M$,}
\end{equation}
the uniqueness of the equilibrium state is also established. Clearly, condition~\eqref{propestrela} holds whenever $f$ is \emph{strongly topologically mixing}, i.e. if for every open set $U\subset M$ there is some $N\in\mathbb N$ such that $f^N(U)=M$. 
Given $c>0$, define
 \begin{equation*}\label{eq.Gc}
 \mathcal H_c=\left\{(f,\phi)\in \mathcal F\times C^\alpha(M): \text{ $\phi$ is $c$-hyperbolic  and $(\ast)$ holds for $f$} \right\}.
 \end{equation*}
 In our first result we obtain the continuous dependence   of the  equilibrium states over  the elements in the family  $\mathcal H_c$. 

\begin{maintheorem}\label{statistical para f}
 $\mathcal H_c$  is equilibrium stable.
\end{maintheorem}

As a byproduct of our method to  prove Theorem~\ref{statistical para f} we obtain in Corollary~\ref{continuidade da pressao na base} the continuity of the topological pressure within the family $\mathcal H_c$.

\subsection{Uniformly expanding maps} Theorem~\ref{statistical para f}   has a simple but interesting consequence for uniformly expanding maps. 
Letting $\mathcal E\subset \mathcal F$ be the set of uniformly expanding maps, for each $f\in\mathcal E$ we have $\Sigma_c(f)=M$ for suitable (locally constant) choices of $c>0$ and norm in $M$. Moreover, it is well-known from the classical theory that  each $(f,\phi)\in \mathcal E\times C^\alpha(M)$ has a unique equilibrium state $\mu_{f,\phi}$ and \eqref{propestrela} holds for all $f\in\mathcal E$.  The next result is then an immediate consequence of Theorem~\ref{statistical para f}.
\begin{maincorollary}\label{co.b}
$\mathcal E\times C^\alpha(M)\ni (f,\phi)\longmapsto \mu_{f,\phi} \in \mathbb P_f(M)$ is continuous.
\end{maincorollary}

Though not explicitly stated anywhere in the literature,  to the best of our knowledge, we believe that Corollary~\ref{co.b} can probably be considered a folklore result by the experts in the field.

\subsection{Skew-product maps}  In the second part of this work we study skew products over non-uniformly expanding maps. 
Let $N$ be a compact metric space with a distance $d_N$ and let $g:M\times N\rightarrow N$ be a continuous map which is a uniform contraction on $N$, i.e. there exists  $0<\lambda<1$ such that for all $x\in M$  and all $y_1,y_2\in N$   we have
\begin{equation}\label{g}
d_N\big(g(x, y_1), g(x,y_2)\big)\leq\lambda d_{N}(y_1,y_2).
\end{equation}  
We assume that there exists some $y_0\in N$ such that $g(x,y_0)=y_0$ for every $x\in M.$ 
We define a family  $\mathcal S$ of skew-product maps $F:M \times N \to M \times N $ such that
    $$   F(x,y)= (f(x),g(x,y))  $$
for all $(x,y)\in M\times N$, where the {\em base dynamics} $f:M\rightarrow M$ belongs to $\mathcal{F}$ and the fiber dynamics $g:M\times N\rightarrow N$ satisfies~\eqref{g}. For notational simplicity, we shall denote the base dynamics  of $F$ by  $b_F$.

Given $c>0$ and $F\in\mathcal S$ we say that  a continuous function $\phi: M\times N \to \mathbb{R}$ is a {\em $c$-hyperbolic potential for $F$} if  the  topological pressure of the system $(F, \phi)$ is equal to the relative pressure on the set~$\Sigma_c(f)\times N$.
Our second main result states the finiteness of equilibrium states for skew-products with respect to hyperbolic potentials.
    
\begin{maintheorem}\label{unicidade}
Let $F\in\mathcal S$ and  let ${\phi}: M\times N \to \mathbb{R}$ be a H\"older continuous and $c$-hyperbolic potential for $F$, for some $c>0$. Then there exists a finite number of ergodic equilibrium states for $(F, {\phi})$.
\end{maintheorem}

We point out that if the fiber $N$ can be decomposed as a finitely union $N=N_1\cup \cdots \cup N_n$ of pairwise disjoint compact sets $N_1, \cdots, N_n$ then the condition $g(x, y_0)=y_0$ for all $x\in M$ can be replaced by $g_i(x, y_i)=y_i$ for all $x\in M$ and some $y_i\in N_i$, $i=1, \cdots, n$.
In fact, since $M\times N$ is a product space and $N=N_1\cup \cdots \cup N_n$ we may define $n$ fiber dynamics $g_i:M\times N_i\to N_i$ by $g_i(x, y)=g(x, y)$ when $y\in N_i$, for each $i=1,\cdots, n$. See Example~\ref{ferradura}. 

The conditions  on the fiber dynamics allow us to enlarge the class of potentials considered in \cite[Theorem 3]{RV16}, where the independence of the stable direction on the potential was required.

Given $c>0$, consider 
 $$\mathcal G_c=\{(F,\phi)\in\mathcal S\times C^\alpha(M\times N):  \text{ $\phi$ is $c$-hyperbolic and \eqref{propestrela} holds for $b_F$} \}.$$
%
%
%
%
%
%
It easily follows from Lemma~\ref{bijection} below that each $F\in\mathcal S$ with \eqref{propestrela}
holding for $b_F$ necessarily has a unique equilibrium state. Our next result establishes the continuity of such equilibria within this family. 
\begin{maintheorem}\label{statistical}
$\mathcal G_c$  is equilibrium stable.
\end{maintheorem}


\section{Preliminaries}\label{hyp pot}

In order to make this work  self contained we present here some basic definitions and results that  will be used in the next sections. We start by the definition of relative pressure through the notion of dynamic balls. This is the approach  from dimension theory and it is very useful to compute the topological pressure of non-compact sets. 

\subsection{Topological pressure} Let $X$ be a compact metric space. Consider $T:X\to X$  and  $\phi:X\to\mathbb{R}$ both continuous.
Given $\delta>0$, $n\in\mathbb{N}$ and $x\in X$, consider the \emph{dynamic ball}  
$$B_{\delta}(x,n)=\left\{y\in X:\;d(T^{j}(x), T^{j}(y))<\delta,\;\mbox{for}\;\; 0\leq j\leq n\right\}.$$
Consider for each  $N\in\mathbb N$
 $$\mathcal{F}_{N}=\{B_{\delta}(x,n):\; x\in X\;\mbox{and}\; n\geq N\}.$$
Given $\Lambda\subset X$ denote by $\mathcal{F}_{N}(\Lambda)$ the finite or countable families of elements in $\mathcal{F}_{N}$ which cover $\Lambda.$
 Define for $n\in\mathbb N$
 $$S_n\phi(x)=\phi(x)+\phi(T(x))+\cdots+\phi(T^{n-1}(x)).
 $$
 and
 $$R_{n,\delta}\phi(x)=\sup_{y\in B_\delta(x,n)}S_n\phi(y).$$
Given a $T$-invariant set $\Lambda\subset X$, not necessarily compact, define for each~$\gamma>0$
$$
m_{T}(\phi, \Lambda, \delta, N,\gamma)=
\inf_{\mathcal{U}\subset\mathcal{F}_{N}(\Lambda)} \left\{\sum_{B_{\delta}(x,n)\in\mathcal{U}}e^{-\gamma n+ R_{n,\delta}\phi(x)}\right\}.
$$
Define
$$m_{T}(\phi, \Lambda, \delta, \gamma)=\lim_{N\rightarrow +\infty}{m_{T}(\phi, \Lambda, \delta, N,\gamma)},$$
and 
$$P_{T}(\phi, \Lambda, \delta)=\inf{\{\gamma >0 \, | \; m_{T}(\phi, \Lambda, \delta,\gamma)=0\}}.$$
Finally, define the \textit{relative pressure} of $\phi$ 
on $\Lambda$ as 
 $$P_{T}(\phi, \Lambda)=\lim_{\delta\rightarrow 0} {P_{T}(\phi, \Lambda, \delta)}.$$
The \emph{topological pressure of $\phi$} is by definition $P_T(\phi,X)$, and it satisfies
\begin{equation}\label{pressure}
P_{T}(\phi)=\sup\left\{ P_{T}(\phi, \Lambda),\, P_{f}(\phi, \Lambda^{c})\right\},
\end{equation}
where $\Lambda^c$ stands for the complement of the set $\Lambda$ on $M$. We refer the reader to~\cite{Pes} for the proof of~\eqref{pressure} and for additional  properties of the pressure. See also \cite[Theorem 4.1]{W} for a proof of the fact that
\begin{equation}\label{pressure2}
P_T(\phi)  = \sup_{\eta\in\mathbb P_T(X)} \left\{ h_{\eta} (T) + \int \phi \, d\eta 
\right\}.  
\end{equation}

Consider $C^0(X)$ the space of continuous functions from $X$ to $\mathbb R$, endowed with the supremum norm.
\begin{Prop}\label{pr.openess}  Let $\Lambda\subset X$ be a $T$-invariant set and  $\phi\in C^{0}(X)$. If $P_T(\phi, \Lambda^c)<P_T(\phi, \Lambda)=P_T(\phi)$, then there exists $\zeta>0$ such that for each $\psi\in C^{0}(X)$ with $\|\psi-\phi\|<\zeta$ we have $P_T(\psi, \Lambda^c)< P_T(\psi,\Lambda)=P_T(\psi).$
\end{Prop}

\begin{proof} Let $\phi\in C^{0}(X)$ satisfy $P_T(\phi, \Lambda^c)<P_T(\phi, \Lambda)=P_T(\phi)$. Fix $\varepsilon_\phi>0$ such that $P_T(\phi, \Lambda^c)<P_T(\phi)-\varepsilon_\phi.$
Recalling that the topological pressure depends continuously on the potential $\phi\in C^0(X)$ (see e.g. \cite[Proposition 10.3.6]{VO}), given $0<\varepsilon<\varepsilon_\phi$ we can find $\zeta>0$ such that for any $\psi\in C^{0}(X)$ with $\|\psi-\phi\|<\zeta$ we have $P_T(\psi)\geq P_T(\phi)-\varepsilon.$
For every $\gamma\in\mathbb{R}$ we have
\begin{eqnarray*}
m_{T}(\psi, \Lambda^c, \delta, N,\gamma)&=&\inf_{\mathcal{U}\subset\mathcal{F}_{N}(\Lambda^c)} \left\{\sum_{B_{\delta}(x,n)\in\mathcal{U}}e^{-\gamma n+ R_{n,\delta}\psi(x)}\right\}\\
&\le&\inf_{\mathcal{U}\subset\mathcal{F}_{N}(\Lambda^c)} \left\{\sum_{B_{\delta}(x,n)\in\mathcal{U}}e^{-\gamma n+\zeta n+ R_{n,\delta}\phi(x)}\right\}\\
&=&m_{T}\left(\phi, \Lambda^c, \delta, N,\gamma -\zeta\right),
\end{eqnarray*}
and therefore 
$$P_T(\psi, \Lambda^c)\leq P_T(\phi, \Lambda^c)-\zeta<P_T(\phi)-\varepsilon_{\phi}- \zeta\leq P_T(\phi)-\varepsilon\leq P_T(\psi).$$
Hence, for any  $\psi\in C^{0}(X)$ with $\|\psi-\phi\|<\zeta$ we have $$P_T(\psi, \Lambda^c)< P_T(\psi,\Lambda)=P_T(\psi),$$
which gives the desired conclusion.   
\end{proof}

Now we introduce  the definition of hyperbolic potentials.
Let $M$ be a compact Riemannian manifold and let $f:M\to M$ be a local $C^1$ diffeomorphism.  Consider the set $\Sigma_c(f)$ as in Subsection~\ref{se.NUE}. We say that   $\phi:M\to\mathbb{R}$ continuous is a {$c$-\em hyperbolic potential} for $f$ if 
        $$P_f(\phi, \left(\Sigma_c(f)\right)^{c})<P_f(\phi, \Sigma_c(f))=P_f(\phi).$$
Proposition~\ref{pr.openess} gives  in particular that the hyperbolic potentials form an open class in the $C^0$-topology.

\subsection{Cones} We will also use a few tools from classic functional analysis that we present here. 
Let $E$ be a  Banach space over  $\mathbb K= \mathbb R$ or $\mathbb C$.  We say that a closed and convex set $\{ 0\} \neq \mathcal{C} \subset E$ is  a \emph{cone} in $E$ if 
\begin{enumerate}
\item $\forall \lambda \geq 0 : \lambda \mathcal{C} \subset \mathcal{C}; $
\item $ \mathcal{C} \cap  \left( - \mathcal{C}\right) = \{0\} .$
\end{enumerate}
A cone $\mathcal{C}$ defines a partial order in $E$  through  the relation  
       $$ x \leq  y \Longleftrightarrow  y - x \in \mathcal C .  $$ 
We say that    a cone $\mathcal{C} \subset E$ is   \emph{normal} if 
$$ \exists \ \gamma \in  \mathbb{R} : 0 \leq x \leq y  \ \Longrightarrow \| x \| \leq  \gamma \| y \| .$$
Assume that $E$ is a Banach space partially ordered by a normal cone $\mathcal{C}$ with  non-empty interior and $T: E \rightarrow  E $ is a bounded linear  operator. We say that $T$ is \emph{positive} (with respect to $\mathcal{C}$) if $T(\mathcal{C}) \subset \mathcal{C}$. Note that if T is positive then $T(x)\leq T(y)$ whenever $x\leq y$. The \emph{spectral radius} of T is defined by
   $$r \left( T \right)=\displaystyle\lim_{n \to \infty} \displaystyle\sqrt[n]{\parallel T^{n} \parallel }   . $$
 The dual space of $E$ is defined by 
   $$E^{\ast}=\{x^{\ast}:E\to \mathbb K\; |\; \text{$x^{\ast}$ is linear and bounded}\}.$$
and the dual operator of $T$ is
$T^{\ast}:E^{\ast}\to E^{\ast} $ defined for all $x^{\ast}\in E^{\ast}$ and  all  $y\in E$ by
$$T^{\ast}(x^{\ast})(y)=x^{\ast}(T(y)).$$ 
Given $x^{\ast}\in E^{\ast}\setminus\{0\}$ and $c\in \mathbb R$, we define the real hyperplane  $$H=H(x^{\ast}, c) =\left\{ y\in E \; | \; \mathcal {R} x^{\ast}(y) = c\right\},$$ 
where  $\mathcal {R}a$ stands for the real part of $a\in\mathbb C$. We say that  $H$ \emph{separates} the sets
  $E_1,E_2\subset E$ if 
  $$
  \sup_{y\in E_1}\mathcal {R}x^*(y)  \leq c \leq   \inf_{z\in E_2}\mathcal {R}x^*(z).
  $$ 
Given a cone $\mathcal C\subset E$, we say that $x^{\ast}\in E^{\ast}$  is a positive functional  (with respect  to $\mathcal C$) if   $\mathcal {R}x^{\ast}(w) \geq 0$ for all $w \in C$.  The dual cone $\mathcal C^{\ast}$ is the cone given by all positives functionals.

A classical result due to Mazur can be stated as follows:   {\it let $E$ be a Banach space and  $C \subset E$ a convex set with  nonempty interior.  If  $C_{1} \neq \emptyset$ is a convex set such that  $C_{1} \cap \inter(C)= \emptyset$, then there exists a closed real hyperplane which separates  $C$ and $C_{1}$.}
For a proof see e.g. \cite[Proposition 7.12]{Deimling}.

We end this section with a useful corollary from Mazur\rq{s} Theorem. This is somewhat a floklore result but  for the sake of completeness we present here its proof. 

\begin{Lema}\label{Mazur}

Let $E$ be a Banach space partially ordered by a normal cone $\mathcal{C}$ with  non-empty interior. Let $T$ be a positive bounded operator on $E$.
Then the spectral radius of $T$  is an eigenvalue of the adjoint  $T^{\ast}$,  associated to an eigenfunctional  $x^{\ast} \in \mathcal{C}^{\ast}$. 
\end{Lema}
\begin{proof}
Defining  $E_{0}= \{ r(T)x - Tx : x \in E \}$, we start by showing that $E_{0}\cap \mbox{int}(\mathcal{C}) = \emptyset$.
Assume by contradiction    that $E_{0}\cap \mbox{int}(\mathcal{C}) \neq \emptyset$. This means that there exist 
 $x \in E$ and $y \in \mbox{int}(\mathcal{C}) $ such that $y= r(T)x - Tx$. Thus there exists  $\delta >0$  such that  $y - \delta (-x)\in\inter (\mathcal C)\subset \mathcal C$, and so $y \ge  \delta (-x) $.
Using that  $T(-x)= r(T)(-x) + y$, we easily deduce  that 
$$T(-x) \geq (r(T)+ \delta)(-x).$$ 
By linearity and positivity of $T$ we obtain for all $n \in \mathbb{N}$
     $$T^{n}(-x) \geq (r(T)+ \delta)^{n}(-x).$$
On the other hand, since we are assuming the cone $\mathcal C$ with non-empty interior, there must be $a\in\mathcal C$ and   $\delta' > 0$  such that  $a + \delta' (-x) \geq  0$. Taking  $K= 1/\delta'$ we obtain $Ka - x \geq 0$ and 
$$T^{n}(-x) + (r(T)+ \delta)^{n}(Ka)\geq (r(T)+ \delta)^{n}(-x + Ka) \ge 0.     $$
Since $\mathcal{C}$ is a normal cone,  by definition there is $\gamma>0$ such that
$$\gamma\| T^{n}(-x) + [r(T) + \delta]^{n}Ka\| \geq  \| [r(T) + \delta]^{n}(-x + Ka) \| .$$ 
It easily follows that for all odd   $n \in \mathbb{N}$  we have
   $$\sqrt[n]{\gamma\| -x\| } \sqrt[n]{\| T^{n}\| }  \geq  [r(T)+ \delta] \sqrt[n]{\| Ka - x\|  - \gamma \| Ka \| } .    $$
taking limit in  $n$ we obtain $r(T)\geq r(T) + \delta $, which is an absurd. Hence, we must have  $E_{0}\cap \inter(\mathcal{C} )= \emptyset$.  

Now since $E_0$ and $\mathcal C$ are convex sets ($E_0$ is actually a vectorial subspace of $E$), 
Mazur\rq{s} Theorem guarantees the existence of a closed real hyperplane separating them. 
This means that there exist  $c \in \mathbb{R}$ and a non-trivial continuous linear  functional $x^{\ast}\in E^*$ such that
               $$\sup_{x \in E_{0}} \mathcal{R} x^{\ast}(x) \leq c \leq \inf _{y \in \mathcal{C} }\mathcal{R} x^{\ast}(y).$$
Since  $0$ belongs to both $E_{0}$ and $ \mathcal{C}$ we must have $c=0$. Moreover, as $E_{0}$ is a vectorial subspace of $E$ we necessarily have $x^{\ast}(x)=0 $ for all $x \in  E_{0}$. This implies that for all $x \in E$ we have 
$$
                       x^{\ast} (r(T)x -Tx) =0 ,
                       $$
                       which is equivalent to 
                       $$
                      T^{\ast} x^{\ast} (x) = r(T)x^{\ast}(x).   
$$
This gives in particular that the spectral radius of $T$  is an eigenvalue of~$T^{\ast}$.
Moreover, since $\mathcal{R}x^{\ast}(y) \geq 0$ for all $y \in \mathcal{C}$, and so  we have  $x^{\ast} \in \mathcal{C}^{\ast}$.
\end{proof}


\subsection{Hyperbolic times}
Consider $f\in\mathcal F$ and $\Sigma_c(f)$ as in Subsection~\ref{se.NUE}.
%
In order to explore the non-uniform expansion on the set $\Sigma_c(f)$ we  use the following concept  introduced in \cite{Alves} and generalized in \cite{ABV}.
We say that $n$ is a \textit{hyperbolic time} for $x$ if 
$$\prod_{j=n-k}^{n-1}\|Df(f^{j}(x))^{-1}\|\leq e^{-ck/2},\quad\text{for all }1\leq k< n.$$
Observe that as we are assuming maps with no critical/singular sets, the definition of hyperbolic times given in \cite[Definition~5.1]{ABV} reduces  to the one  we present here.
Condition \eqref{limsup} of non-uniform expansion  is enough to guarantee the existence of infinitely many hyperbolic times for the points in $\Sigma_c(f)$. For a proof of the next result see \cite[Lemma~5.4]{ABV}.

\begin{Lema}\label{le.infinite}
Each $x\in \Sigma_c(f)$ has infinitely many hyperbolic times.
\end{Lema}

The next result shows  that the iterates of a map at hyperbolic times behave locally as  uniformly expanding maps. We refer the reader to \cite[Lemma~5.2 ]{ABV} for the proof of the first item and \cite[Corollary~5.3]{ABV}  for the second one, with $\phi$ playing the role of $\log|\det Df|$ in the latter case.
\begin{Lema}\label{distortion} There exist $K,\delta_1>0$ such that for every  $0<\varepsilon\le\delta_1$ and every hyperbolic time~$n$ for $x\in M$  the dynamic ball $B_{\varepsilon}(x,n)$ is mapped diffeomorphically under $f^{n}$ onto the ball $B(f^{n}(x),\varepsilon)$. Moreover, for all H\"older continuous potential~$\phi$ and all $y, z\in B_{\varepsilon}(x,n)$ we have
\begin{enumerate}
\item $
d(f^{n-j}(y), f^{n-j}(z))\leq e^{-cj/4}d(f^{n}(y), f^{n}(z))
$ for each $1\le j\le n$;
\item $\displaystyle
\frac1{K}\le  {e^{S_{n}\phi(y)-S_{n}\phi(z)}} \le K.
 $
\end{enumerate}
\end{Lema}


\begin{Remark}
The constant $\delta_1>0$ in the previous lemma can be taken uniform in a neighborhood of $f\in\mathcal F$. Actually, $\delta_1$ is obtained  as the conclusion of the Claim in the proof of \cite[Lemma~5.2]{ABV} with $\sigma=e^{-c/2}$, where the only requirement is that
\begin{equation}
\label{e.littleloss}
\|Df(y)^{-1}\| \le e^{c/4} \|Df(f^{n-j}(x))^{-1}\|,
\end{equation}
for any $1 \le j < n$ and any $y$ in the ball of
radius $2\delta_1 e^{-cj/4}$ around $f^{n-j}(x)$. In our situation, where  $f$ is a local  diffeomorphism  on a compact manifold, we can clearly choose $\delta_1>0$ sufficiently small so that
\begin{equation*}
\label{e.littleloss2}
\|Df(y)^{-1}\| \le e^{c/4} \|Df(z)^{-1}\|
\end{equation*}
whenever $d(y,z)<2\delta_1$. 
This choice of $\delta_1$ is naturally sufficient for~\eqref{e.littleloss} to hold and can be taken uniform in a neighborhood of $f$ in the $C^1$ topology. 
\end{Remark}


 \subsection{The reference measure}      Let $C^{0}(M)$ be the space of continuous functions $\phi:M\to \mathbb R$ endowed with the sup norm. Fix $c>0$ and define  for $(f,\phi)\in \mathcal H_c$ the {\emph{transfer operator}} $$\mathcal{L}_{f,\phi} : C^{0}(M) \rightarrow C^{0}(M),$$    associating to each $\varphi \in C^{0}(M)$ the continuous  function
$ \mathcal{L}_{f,\phi} (\varphi) \colon M\to \mathbb{R}$ defined by
$$\label{optransf}
\mathcal{L}_{f,\phi} \varphi \left(x\right) = \displaystyle\sum _{y  \in \, f^{-1}\left(x\right)} e^{\phi(y)} \varphi(y). 
$$
We clearly have that $ \mathcal{L}_{f,\phi} $ is a bounded linear operator. Moreover, 
considering $C^0(M)$ ordered by the cone $\mathcal C$  of nonnegative functions, we have that $ \mathcal{L}_{f,\phi} $ is a positive operator. 
For each $n \! \in\! \mathbb{N}$ we have  
$$\label{iteradostransf}
\mathcal{L}_{f,\phi}^{n} \varphi (x) = \displaystyle\sum _{y  \in \, f^{-n}\left(x\right)} e^{S_{n}\phi(y)} \varphi \left( y \right),
$$
where $S_{n}\phi$ denotes the Birkhoff sum $$S_{n}\phi(x)= \displaystyle\sum_{j=0}^{n-1} \phi\big(f^{j}(x)\big).$$
Using that $\|\mathcal{L}_{f,\phi}^{n} \|=\|\mathcal{L}_{f,\phi}^{n}1 \|$ for all $n\ge1$,  we can easily deduce that the spectral radius $\lambda_{f,\phi}$ of $\mathcal{L}_{f,\phi}$ satisfies
\begin{equation}\label{eq.radius}
\deg(f)e^{\inf \phi}\leq  
\lambda_{f,\phi}
\leq\textbf{\rm{deg}}(f)e^{\sup\phi}.
\end{equation}
As we are considering $\mathcal{L}_{f,\phi} : C^{0}(M) \rightarrow C^{0}(M)$, by Riesz-Markov Theorem, we may think of  its dual operator $\mathcal{L}_{f,\phi}^{\ast}: \mathbb{P}(M) \to \mathbb{P}(M)$. Moreover,  we have
$$\int \varphi \ d\mathcal{L}_{f,\phi}^{\ast}\eta  = \int  \mathcal{L}_{f,\phi}(  \varphi ) \ d\eta , $$
for every $\varphi \in  C^{0}(M) $ and every $\eta \in \mathbb{P}(M)$. It follows from Lemma~\ref{Mazur} that
for each $(f,\phi)\in \mathcal H_c$  there exists a probability measure~$\nu_{f,\phi}$ satisfying $$\mathcal{L}_{f,\phi}^{\ast}\nu_{f,\phi}= \lambda_{f,\phi} \nu_{f,\phi} ,$$
where $\lambda_{f,\phi}$ is the spectral radius of $\mathcal{L}_{f,\phi} $.

The  next result gives   a \emph{Gibbs property} for the measure $ \nu_{f,\phi} $ at hyperbolic times. 
Our proof follows closely the proof of \cite[Proposition 5.2]{RV16}, where similar conclusions hold for the spectral radius  $\lambda_{f,\phi}$ and the respective eigenmeasure $\nu_{f,\phi}$.

\begin{Lema}\label{conforme}
Given $(f,\phi)\in\mathcal H_c$ 
and assuming $\mathcal{L}_{f,\phi}^{\ast}\nu= \lambda \nu$ for some  $\lambda \in \mathbb R$ and $\nu\in\mathbb P(M)$, then $\supp(\nu)=M$. Moreover, 
for all $\varepsilon\leq\delta_1$ there exists $C=C(\varepsilon)>0$ such that if  $n$ is a hyperbolic time for $x\in M$, then
\begin{eqnarray*} \label{Gibbs}
  C^{-1} \leq  \frac{\nu(B_\varepsilon(x, n))}{\exp(S_n\phi(y) - n\log\lambda)}  \leq C
	\end{eqnarray*}
	for all $y\in B_\varepsilon(x, n)$.
\end{Lema}
\begin{proof} Since $\nu$ is an eingenmeasure associated to $\lambda$, it follows that if $A$ is a Borel set such that $f|_A$ is injective, then  for any sequence $\{\psi_{s}\}_s$ in $ C^0(M)$  which converges to the characteristic function $\chi_{A}$ of $A$ for  $\nu$ almost every point we have
$$\lambda\nu(e^{-\phi}\psi_{s})=\mathcal{L}^{*}_{\phi}\nu(e^{-\phi}\psi_{s})=\int_{M}{\mathcal{L}_{\phi}( e^{-\phi}\psi_{s}})d\nu\stackrel{s\to\infty}\longrightarrow \nu(f(A)).$$
As the first member converges to $\int_{A}\lambda e^{-\phi}d\nu$, we conclude that
$$\nu(f(A))=\int_{A}{\lambda e^{-\phi}}d\nu.$$
In particular, if $f^k\vert_A$ is injective  it follows that $$\nu(f^{k}(A))=\int_{A}\lambda^{k}e^{-S_{k}\phi}d\nu.$$
Let $U\subset M$ be an arbitrary open set. From the density of the pre-orbit $\{f^{-n}(x)\}_{n\geq 0}$ for every point $x\in M$ we have the inclusion $M\subset\bigcup_{k\in\mathbb{N}}f^{k}(U)$.
Since each $f^{k}$ is a local homeomorphism we can decompose $U$ into subsets $V_{i}(k)\subset U$ such that $f^{k}|_{V_{i}(k)}$ is injective. Using the fact that $\nu$ is an eingenmeasure, we deduce that 
\begin{eqnarray*}
1=\nu(M)\leq\sum_{k}\nu(f^{k}(U))
&\leq&\ds\sum_{k}\sum_{i}\int_{V_{i}(k)}{\lambda^{k}e^{-S_{k}\phi(x)}}d\nu\\ 
&\leq&\ds\sum_{k}\lambda^{k}\ds\sum_{i}\sup_{x\in V_{i}(k)}(e^{S_{k}\phi(x)})\nu(V_{i}(k)).
\end{eqnarray*}
Hence, there exists some $V_{i}(k)\subset U$ such that $\nu(U)\geq\nu(V_{i}(k))>0$. Thus we have $\supp(\nu)=M$.

Consider now $x\in M$ and $n\in \mathbb N$   a hyperbolic time for $x$. Since $f^{n}$ maps the hyperbolic dynamic ball $B_{\varepsilon}(x,n)$ homeomorphically onto the ball $B(f^{n}(x),\varepsilon)$, it  follows that there exists a uniform constant $\gamma_\varepsilon>0$ depending on the radius $\varepsilon$ of the ball such that
\begin{eqnarray*}
\gamma_\varepsilon\leq\nu(B(f^{n}(x),\varepsilon))=\nu(f^{n}(B_{\varepsilon}(x,n)))=\int_{B_{\varepsilon}(x,n)}{\hspace{-0.4cm}\lambda^{n}e^{-S_{n}\phi(z)}}\,d\nu\leq1.
\end{eqnarray*} 
Using the H\"older continuity of the potential $\phi$, we apply the distortion control on hyperbolic times (Lemma \ref{distortion}) to obtain
\begin{eqnarray*}
\gamma_\varepsilon\leq\int_{B_{\varepsilon}(x,n)}\hspace{-0.4cm}\lambda^{n}e^{-S_{n}\phi(y)} \frac{e^{-S_{n}\phi(z)}}{e^{-S_{n}\phi(y)}}\,d\nu
\!\!\!&\leq&\!\!\! Ke^{-S_{n}\phi(y)+n\log\lambda}\nu(B_{\varepsilon}(x,n))\\
\!\!\!&\leq&\!\!\! K\int_{B_{\varepsilon}(x,n)}\hspace{-0.4cm}\lambda^{n}e^{-S_{n}\phi(y)}\,d\nu\leq K.
\end{eqnarray*}
Hence, there exists $C=C(\varepsilon)>0$ such that
$$C^{-1} \leq  \frac{\nu(B_\varepsilon(x, n))}{\exp(S_n\phi(y) - n\log\lambda)}  \leq C$$ 
for all $y\in B_{\varepsilon}(x,n)$. 
\end{proof}


\begin{Lema}\label{conformeunico}
For each $(f,\phi)\in \mathcal H_c$  we have   $\lambda_{f,\phi}=e^{P_f(\phi)}$ and this is the only real eigenvalue  of $\mathcal{L}_{f,\phi} ^*$. 
\end{Lema}

\begin{proof}
The fact that $\lambda_{f,\phi}=e^{P_f(\phi)}$ has been proved in \cite[Proposition~5.1]{RV16}.
Assume now that $\mathcal{L}_{f,\phi}^{\ast}\nu= \lambda \nu$ for some  $\lambda \in \mathbb R$ and $\nu\in\mathbb P(M)$. As every point $x\in \Sigma_c(f)$ has infinitely many hyperbolic times, for $\varepsilon>0$ small we can fix $N>1$ sufficiently large such that
$$\Sigma_c(f)\subset\bigcup_{n\geq N}\bigcup_{x\in H_n}B_{\varepsilon}(x,n),$$
where $H_n$ denotes the set of points that have $n$ as a hyperbolic time. 
Using the fact that each $f^n(B_\varepsilon(x,n))$ is the ball $B(f^n(x),\varepsilon)$ in $M$ and Besicovitch Covering Lemma it is not difficult to  see that there exists a countable family $F_{n}\subset H_{n}$ such that every point $x\in H_{n}$ is covered by at most $d=\dim(M)$ dynamical balls $B_{\varepsilon}(x,n)$ with $x\in F_{n}.$ Hence 
$$\mathcal{F}_{N}=\left\{B_{\varepsilon}(x,n) : x\in F_{n}\;\textbf{\rm{and}}\;n\geq N\right\}$$ 
is a countable open covering of $\Sigma_c(f)$ by hyperbolic dynamic balls with diameter less than $\varepsilon>0$.

Take now any $\gamma>\log\lambda$. Recalling the definition of relative pressure given in Section~\ref{hyp pot} we apply Lemma~\ref{conforme} to each element in $\mathcal{F}_{N}$ to deduce that for some $\widetilde C=\widetilde C(\varepsilon)>0$ we have
     \begin{equation*}\label{eq.tilde}
     \sum_{B_{\varepsilon}(x,n)\in \mathcal{F}_{N}}   e^{ - n\log\lambda+R_{n,\varepsilon}\phi(x)}  \leq \widetilde{C}.
     \end{equation*} 	
     Hence
\begin{eqnarray*}
\sum_{B_{\varepsilon}(x,n)\in \mathcal{F}_{N}}   e^{-\gamma n+R_{n,\varepsilon}\phi(x) }\leq\widetilde C\sum_{n\geq N}e^{-(\gamma-\log\lambda)n}
\leq \widetilde C e^{-(\gamma-\log\lambda)N}
\end{eqnarray*}
Taking  limit in $N$ we obtain 
$$m_{f}(\phi, \Sigma_c(f), \varepsilon, \gamma)=\lim_{N\rightarrow +\infty}{m_{f}(\phi, \Sigma_c(f), \varepsilon, N,\gamma)}=0.$$
As this holds for arbitrary $\varepsilon>0$ and $\gamma>\log\lambda$ we necessarily have $$P_{f}(\phi,\Sigma_c(f))\leq\log\lambda.$$
Now, since $\phi$ is a hyperbolic potential and $\lambda_{f,\phi}$ is the spectral radius of $\mathcal{L}_{f,\phi}^{\ast}$ we  get $$\log\lambda\leq \log\lambda_{f,\phi}=P_f(\phi)=P_{f}(\phi,\Sigma_c(f))\leq\log\lambda,$$
thus having proved the result.
%
%
	\end{proof}

\section{Equilibrium Stability}  \label{stability}

In this section we finish the proof of Theorem~\ref{statistical para f}. 
Let  $\mathcal{B}(C^\alpha(M)) $ be the space of bounded linear maps from   $C^\alpha(M)$ to $C^\alpha(M)$, considering the norm $\|\quad\|_\alpha$ in the first space and  $\|\quad\|_0$ in the second one.
 As we have the functions $f$ and $\phi$ H\"older continuous for $(f,\phi)\in\mathcal H_c$, one easily sees that 
  $$\mathcal L_{f,\phi}\left(C^\alpha(M)\right)\subset C^\alpha(M).$$
We introduce  a map
$$\Gamma: \mathcal H_c \longrightarrow  \mathcal{B}(C^\alpha(M)) ,$$
assigning to each $(f,\phi)\in\mathcal H_c$ the restriction  of $\mathcal{L}_{f, \phi}$ to  $C^\alpha(M)$.
In the next result shall we deduce the continuity of $\Gamma$, which obviously implies the weaker pointwise convergence:  for any $\psi\in C^\alpha(M)$ we have $\mathcal L_{f_n,\phi_n}(\psi)$ converging to $\mathcal L_{f,\phi}(\psi)$ in $C^0$-norm. This pointwise convergence actually what we need for for what comes next. However, as the proof of this statement is essentially the same of the weaker pointwise convergence we leave here this   stronger version.

\begin{Lema}\label{C1 continuity} 
$\Gamma$ is continuous.
\end{Lema}

\begin{proof} Given $(f, \phi)\in \mathcal H_c$, let $(f_n, \phi_n)_n$ be an arbitrary sequence in $\mathcal H_c$ converging to $(f, \phi)$. For each  $x\in M$ and $i=1,\dots, \deg(f)$ denote by $y_i$ the preimage of $x$ under $f$. Since $\deg(f_n)=\deg(f)$ for all $n\in\mathbb{N}$, it follows that there exists a unique point $y_{i, n}$, preimage of $x$ under~$f_n$,  in a neighborhood  of $y_i$ for every  $i=1, \dots, \deg(f)$. We point out that the sequence $(y_{i, n})_n$ converges to $y_i$ when $n$ goes to infinity. In fact, let $z$ be an accumulation point of $(y_{i, n})_n$. From the uniform convergence of $(f_n)_n$ to $f$ we have that the sequence $(f_{n}(y_{i,n}))_n$ converges to $f(z)$. Since  $f_{n}(y_{i,n})=x$ for all $n$, we conclude that $z=y_i.$

By definition of the operator norm in $\mathcal{B}(C^\alpha(M))$ we have
\begin{eqnarray*}
\left\|\mathcal{L}_{f, \phi}-\mathcal{L}_{f_n, \phi_n}\right\|
&=& \sup_{\left\|\varphi\right\|_\alpha\leq 1} \|\mathcal{L}_{f, \phi} (\varphi) - \mathcal{L}_{f_n, \phi_n} (\varphi) \|_0\\
& = & \sup_{\left\|\varphi\right\|_\alpha\leq 1}\,  \sup_{x\in M}\sum _{i=1}^{\deg(f)} |e^{\phi(y_i)} \varphi (y_i) - e^{\phi_n(y_{i,n})} \varphi (y_{i,n})|\\
&\leq& \sup_{\left\|\varphi\right\|_\alpha\leq 1}\,  \sup_{x\in M}\!\sum _{i=1}^{\deg(f)}|\varphi (y_i)||e^{\phi(y_i)} - e^{\phi_n(y_{i,n})}|\\
&&+ \sup_{\left\|\varphi\right\|_\alpha\leq 1}\,\sup_{x\in M}\sum _{i=1}^{\deg(f)}|e^{\phi_n(y_{i,n})}||\varphi (y_i) - \varphi (y_{i,n})|.
\end{eqnarray*} 
Since  $(y_{i, n})_n$ converges to $y_i$ and $(\phi_n)_n$ converges uniformly to $\phi$ we have that each term in the last inequality converges uniformly to zero for all $\varphi\in C^{\alpha}(M)$ with $\left\|\varphi\right\|_\alpha\leq 1.$ 
\end{proof}

Let now $(f_n, \phi_n)_n$ be a sequence in $\mathcal H_c$ converging to $(f, \phi)\in \mathcal H_c$. For notational simplicity, for each $n\geq 0 $  let $\lambda_n$ be the spectral radius of $\mathcal{L}_{f_n, \phi_n}$ and $\lambda$ be the spectral radius of $\mathcal{L}_{f, \phi}$. In the next result we use in particular  the continuity of   $\Gamma$ to prove the convergence of the spectral radius.

\begin{Lema} 
 The sequence $(\lambda_n)_n$ converges to $\lambda$.
\end{Lema}

\begin{proof}
By \eqref{eq.radius} we have for all $n\ge 1$
\begin{eqnarray*}\label{cota autovalor}
 \deg(f_n)e^{\inf \phi_n} \leq \lambda_n \leq \deg(f_n)e^{\sup \phi_n}. 
\end{eqnarray*}   
In particular, the convergence of $(f_n, \phi_n)$ to $(f,\phi)$ implies that the sequence   $(\lambda_n)_n$ admits some accumulation point   $\bar{\lambda}\in\mathbb R$
satisfying
 $$ \deg(f)e^{\inf \phi} \leq \bar{\lambda} \leq \deg(f)e^{\sup \phi}.$$ 
Now we are going to prove that $\bar{\lambda}$ coincides with $\lambda.$      
  By Lemma~\ref{conformeunico} it is enough to show that $\bar\lambda$ is an eigenvalue of $\mathcal{L}_{f, \phi}^{\ast}$.
Taking subsequences, if necessary, we may assume that there is $\nu\in \mathbb P(M)$ such that 
 $$\nu_n\stackrel{w*}\longrightarrow\nu\quad\text{and}\quad \lambda_n\longrightarrow\bar \lambda.$$
 Let us show that
 $
 \mathcal{L}_{f, \phi}^{\ast}(\nu)=\bar\lambda\nu.
 $
Since $C^\alpha(M)$ is dense in $C^0(M)$, it is enough to see that 
 $$
 \mathcal{L}_{f, \phi}^{\ast}(\nu)(\psi)=\bar\lambda\nu(\psi),\quad \forall\, \psi\in C^\alpha(M).
 $$
Using the continuity of $\Gamma$ and $\nu$, and the fact that $\nu_n\to \nu$,  we may deduce for each $\psi\in C^\alpha(M)$
 \begin{eqnarray*}
\mathcal{L}_{f, \phi}^{\ast}(\nu)(\psi) 
&=&\nu\left( \mathcal{L}_{f, \phi}(\psi)\right)\\ 
&=&\nu\left(\lim_{n\to\infty} \mathcal{L}_{f_n, \phi_n}(\psi)\right)\\
&=&\lim_{n\to\infty}\nu\left( \mathcal{L}_{f_n, \phi_n}(\psi)\right) \\ 
&=&\lim_{n\to\infty}\nu_n\left( \mathcal{L}_{f_n, \phi_n}(\psi)\right) \\ 
&=&\lim_{n\to\infty} \mathcal{L}_{f_n, \phi_n}^{\ast}(\nu_n)(\psi)\\
&=& \displaystyle \lim_{n \to +\infty} {\lambda_n} \nu_n(\psi)\\ 
&=&  \bar{\lambda} \nu (\psi). 
\end{eqnarray*}
%
%
As  Lemma~\ref{conformeunico} assures that $\lambda$    is the only   eigenvalue of $\mathcal{L}_{f, \phi}^*$ 
we necessarily have $\bar{\lambda} = \lambda$.
\end{proof}

Recalling   that $P_f(\phi)=\log\lambda$, as a consequence of the  previous result we obtain the continuity of the topological pressure on the set~$\mathcal{H}_c$.

\begin{Cor} \label{continuidade da pressao na base}
The  function $$
\begin{array}{cccc}
\mathcal H_c  \ & \longrightarrow  \ & \mathbb{R} \\
    (f,\phi) \ & \longmapsto  \ & P_f(\phi)
\end{array}
$$
is continuous.
\end{Cor}

Now we are able to prove  the equilibrium stability of the family $\mathcal H_c$. 
%
Consider as before  $(f_n, \phi_n)_n$  a sequence in $\mathcal H_c$ converging to $(f, \phi)\in \mathcal H_c$. For each $n\in\mathbb{N}$, let $\mu_n$ be an equilibrium state for $(f_n, \phi_n)$. We are going to show that any accumulation point $\mu_0$ of the sequence $(\mu_n)_n$ is an equilibrium state for $(f, \phi).$ We start with the invariance.

\begin{Lema}$\mu_0$ is an $f$-invariant measure.
\end{Lema}
 \begin{proof} Since each $\mu_n$ is $f_n$-invariant, for any continuous function $\varphi:M\rightarrow\mathbb{R}$
 we have
$$\int \varphi\circ f_{n}\,d\mu_n=\int \varphi \,d\mu_n \longrightarrow \int \varphi\,d\mu_0 \quad \mbox{when} \quad n\to +\infty.$$
Hence to verify the $f$-invariance of $\mu_0$ it suffices to prove that 
$$\int \varphi\circ f_{n}\,d\mu_n \longrightarrow \int \varphi\circ f\,d\mu_0 \quad \mbox{when} \quad n\to +\infty.$$
For each $n\in\mathbb{N}$ we may write the inequality
\begin{eqnarray*}
\left|\int{\varphi\circ f_n}\,d\mu_n-\int{\varphi\circ f}\,d\mu_0\right|&\leq&\left|\int{\varphi \circ f_n}\,d\mu_n-\int{\varphi\circ f}\,d\mu_n\right|\\
&+&\left|\int{\varphi\circ f}\,d\mu_n-\int{\varphi\circ f}\,d\mu_0\right|.
\end{eqnarray*}
Combining the convergence of $f_n$ to $f$ and the fact that $\mu_0$ is an accumulation point of the sequence $(\mu_n)_n$,  we have that each term in the sum above is close to zero for $n$ sufficiently large. This implies that $\mu_0$ is $f$-invariant.
\end{proof}

Now take $\delta_1>0$ as Lemma~\ref{distortion}. Consider $\mathcal{P}$ a finite partition of $M$ with diameter smaller than $\delta_{1}/2$ with $\mu_0(\partial\mathcal P)=0$. For all $n\ge0$ and all $m\geq 1$ define the partition $\mathcal{P}_{m}^n$ by
$$\mathcal{P}_{m}^{n}=\left\{P^n_{m}=P_{i_{0}}\cap\cdots\cap f_n^{-(m-1)}(P_{i_{m-1}})\,;\, P_{i_{j}}\in \mathcal{P}, 0\leq j\leq m-1\right\}.$$
Given $x\in M$, define also 
${P}_{m}^n(x)$ as the element $P_{m}^n\in \mathcal P^n_m$ such that $x\in P_m^n$. Note that by definition the sequence $\{P^{n}_{m}(x)\}_{m\geq 1}$ is non-increasing in $m$, meaning that
$$P^{n}_{m+1}(x)\subset P^{n}_{m}(x),\quad\text{for all $m\geq 1$.}$$

\begin{Lema}\label{le.diam}
For all  $x\in \Sigma_c(f_n)$ and all $n\ge 0$ the diameter of $P^{n}_{m}(x)$ goes to zero when $m$ goes to infinity. 
\end{Lema}

\begin{proof}
It follows from Lemma~\ref{distortion} that if $m=m(x, n)$ is a hyperbolic time for $x\in \Sigma_c(f_n)$, then   $\diam( P^{n}_{m}(x))\leq e^{-cm}\delta_1$. Since every point $x\in \Sigma_c(f_n)$ has infinitely many hyperbolic times we conclude that  $\diam(P^{n}_{m}(x))\rightarrow 0$ when $m$ goes to infinity for every $x\in \Sigma_c(f_n)$ and all $n\ge 0.$
\end{proof}

Now we are ready to finish  the proof of Theorem~\ref{statistical para f}. By \cite[Lemma~4.3]{RV16} we have $\mu_n(\Sigma_c(f_n))=1$ for all $n\ge 1$. Then 
Lemma~\ref{le.diam} gives that $\mathcal P$ is a generating partition for all $(f_n,\mu_n)$. 
As  a straightforward application of \cite[Theorem 11]{Ar} and Kolmogorov-Sinai Theorem we have 
\begin{equation}\label{eq.limsup}
\limsup_{n\to\infty} h_{\mu_n}(f_n)\le h_{\mu_0}(f).
\end{equation}
%
Using    the continuity of the topological pressure (Corollary~\ref{continuidade da pressao na base}) and the fact that $\mu_n$ is an equilibrium state for $(f_n, \phi_n)$,  
 we obtain 
$$P_f(\phi)=\lim_{n\to +\infty}\! P_{f_n}(\phi_n)=\lim_{n\to +\infty}\! \left(h_{\mu_n}(f_n)+\!\int\! \phi_n\, d\mu_n \right)\!\leq h_{\mu_0}(f)+\int\! \phi\, d\mu_0 .$$
This clearly implies that $\mu_0$ is an equilibrium state for $(f,\phi)$.
In the last inequality we have used that   $\mu_n$ converges to $\mu$ in the weak* topology and converges $\phi_n$ to $\phi$ in the $C^0$ norm.


\section{Skew products}\label{est eq}
In this section we prove the existence of finitely many ergodic  equilibrium states for skew products with respect  to hyperbolic H\"older continuous potentials and its stability, namely Theorem~\ref{unicidade} and Theorem~\ref{statistical}.
We point out that Theorem~\ref{unicidade} enlarge the class of potentials considered  in \cite{RV16}, where it is also required that the potential does not depend on the stable direction. Here we prove that this condition is not necessary.

First we extend the definition of hyperbolic potentials for skew products. 
Let $F:M\times N \to M\times N$ be the skew product defined in Section~\ref{Results}.  Recall that a continuous function $\phi:M\times N\rightarrow \mathbb{R}$ is a $c$-{\em hyperbolic potential} for $F$~if 
$$P_F(\phi, \left(\Sigma_c(f)\right)^c\times N)<P_F(\phi,\Sigma_c(f)\times N)=P_F(\phi),$$
where $\Sigma_c(f)$ as in Subsection~\ref{se.NUE}.
In the next result we  show that for every H\"older continuous potential which is hyperbolic for $F$ its possible to construct a potential homologous to it which does not depend on the stable direction and remains hyperbolic for $F$. 
We say that two potentials  $\bar{\phi},\phi: M\times N\to \mathbb R$ are  \emph{homologous} if there is a continuous function $u:M \times N \to\mathbb{R}$ such that $\bar{\phi}=\phi-u+u\circ F$;

\begin{Prop}\label{homologo} Let $\phi: M\times N \to \mathbb{R}$ be a H\"older continuous potential. There exists a H\"older continuous potential $\bar{\phi}: M\times N \to \mathbb{R}$  not depending~on the stable direction such that:
\begin{enumerate}
	\item $\bar{\phi}$ is homologous to $\phi$;
	\item if $\phi$ is   $c$-hyperbolic, then $\bar{\phi}$ is   $c$-hyperbolic;
	\item $P_F(\bar\phi)=P_F(\phi)$;
	\item $(F,\phi)$ and $(F,\bar\phi)$ have the same equilibrium states.
\end{enumerate}
\end{Prop}

\begin{proof} Let $y_0\in N$ be the fixed point of the fiber dynamics. Define a function $u: M\times N\rightarrow\mathbb{R}$ by $$u(x, y)=\sum_{j=0}^{\infty}(\phi\circ F^{j}(x, y)-\phi\circ F^{j}( x,y_0)).$$
Since $(x, y)$ and $(x, y_0)$ are in the same stable direction, for every $(x, y)\in M\times N $ and all $j>0$ we have  
$$ d(F^{j}(x, y), F^{j}(x, y_0)) \leq \lambda^j,$$
where $\lambda$ is the contraction rate of the fiber dynamics (see Section~\ref{Results}).
Using  the last inequality and the $\theta$-H\"older continuity of $\phi$ we obtain
\begin{eqnarray*}
u(x, y)&=&\sum_{j=0}^{\infty}\phi\circ F^{j}(x, y)-\phi\circ F^{j}( x,y_0)\\
&\leq& C \sum_{j=0}^{\infty} d(F^{j}(x, y), F^{j}(x, y_0))^\theta\\
& \leq& C \sum_{j=0}^{\infty} \lambda^{\theta j}.
\end{eqnarray*}
Thus it follows that $u$ is well defined and it is a continuous function. 
 
Defining the potential $\bar{\phi}:M\times N \rightarrow\mathbb{R}$ by $\bar{\phi}:=\phi-u+u\circ F$ we have that $\bar{\phi}$ is a continuous function homologous to $\phi$. Moreover, we can write
\begin{eqnarray*}
\bar{\phi}(x, y)&=&\phi(x, y)-u(x, y)+u\circ F(x, y)\\
&=&\phi(x, y)-\sum_{j=0}^{\infty}\phi\circ F^j(x, y)-\phi\circ F^j(x, y_0)\\
&&+\sum_{j=0}^{\infty}\phi\circ F^{j+1}(x,y)-\phi\circ F^{j}(f(x), y_0)\\
&=&\phi(x, y_0)+\sum_{j=0}^{\infty}\phi\circ F^{j+1}(x, y_0)-\phi\circ F^{j}(f(x), y_0)
\end{eqnarray*}
which implies that $\bar{\phi}$ does not depend on the stable direction.

Since we are assuming that $g(x,y_0)=y_0$ for every $x\in M$, from the last equality we obtain  
\begin{align}
\bar{\phi}(x, y)&=\phi(x, y_0)+\sum_{j=0}^{\infty}\phi\circ F^{j+1}(x, y_0)-\phi\circ F^{j}(f(x), y_0)\nonumber\\
&=\phi(x, y_0)+\sum_{j=0}^{\infty}\phi\circ F^{j}(f(x), g_x(y_0))-\phi\circ F^{j}(f(x), y_0)\nonumber\\
&=\phi(x, y_0).\label{eq.barraphi}
\end{align}
This gives in particular that $\bar{\phi}$ is H\"older continuous and it is a $c$-hyperbolic potential if $\phi$ is $c$-hyperbolic.

The third item is a consequence of \eqref{pressure2} and the fact that $\phi$ and $\bar\phi$ have the same integral with respect to any $F$-invariant probability measure. The fourth item is a consequence of the third one.
\end{proof}

Let $\phi: M\times N \to \mathbb{R}$ be a $c$-hyperbolic H\"older continuous potential  for~$F$. 
According to Proposition~\ref{homologo} there exists a H\"older continuous potential $\bar{\phi}: M\times N \to \mathbb{R}$ that is homologous to $\phi$ and not depending on the stable direction. Fixing any point $z\in N$, the potential $\bar{\phi}$ induces a 
H\"older continuous potential  $\varphi:M\rightarrow\mathbb{R}$ given by 
\begin{equation}\varphi(x)=\bar{\phi}(x, z).\label{eq.fizinho}
\end{equation} 
This means that considering $\pi:M\times N\rightarrow M$ the natural projection onto $M$ defined by $\pi(x, y)=x$ for every $(x,y)\in M\times N$, we have 
 \begin{equation}\label{eq.phis}
 \bar \phi =\varphi\circ\pi.
 \end{equation}
Moreover,  $\pi$ is a continuous semiconjugacy between $F$ and $f$. Using~\eqref{pressure2} it is straightforward to check  that
\begin{equation}\label{eq.pressoes}
P_{f}(\varphi)\leq P_{F}(\bar{\phi}).
\end{equation}
Now we are going to prove that there exists a bijection between the sets of equilibrium states for $(F,  {\phi})$ and for $(f, \varphi)$. For this we will use the following result due to Ledrappier and Walters, whose proof can be found in~\cite{Led}.

\begin{Teo}[Ledrappier-Walters Formula] 
Let $\tilde{X}$, $X$ be compact metric spaces and let $\tilde {T}:\tilde{X}\rightarrow \tilde{X}$, $T:X\rightarrow X$, $\tilde{\pi}:\tilde{X}\rightarrow X$ be continuous maps such that $\tilde{\pi}$ is surjective and $\tilde{\pi}\circ\tilde{T}=T\circ\tilde{\pi}$. Then
$$\sup_{\tilde\mu;\tilde{\pi}_{\ast}\tilde{\mu}=\mu}h_{\tilde{\mu}}(\tilde T)=h_{\mu}(T)+\int{h_{top}(\tilde{T}, \tilde{\pi}^{-1}(y))\,d\mu(y)}.$$
\end{Teo}

\begin{Lema}\label{bijection} If $\mu\in\mathbb{P}_f(M)$ is  ergodic, then there exists a unique ergodic measure $\tilde{\mu}\in\mathbb{P}_F(M\times N)$ such that $\mu=\pi_{*}\tilde{\mu}.$ Moreover, $\mu$ is an equilibrium state for  $(f, \varphi)$ if and only if $\tilde{\mu}$ is an equilibrium state for $(F,  {\phi}).$
\end{Lema}

\begin{proof} Given $\mu\in\mathbb{P}_f(M)$ the existence of $\tilde{\mu}\in\mathbb{P}_F(M\times N)$ such that $\mu=\pi_{*}\tilde{\mu}$ is a classical result that we sketch as follows. Consider the functional
$$L(h\circ\pi)= \int h\, d\mu,$$
defined on the closed subspace of observables in $C^{0}(M\times N)$ which are constant on   fibers of the skew-product. By the Hahn-Banach Theorem we can can extend $L$ to $C^{0}(M\times N)$. Since $L$ is positive and $L(1)=1$ it can be identified with a probability measure $\nu\in\mathbb{P}(M\times N)$. Consider   an accumulation point
$$\tilde{\mu}=\lim_{k\to\infty}\frac{1}{n_k}\sum_{j=0}^{n_k-1}(F^{j})_{*}\nu.$$
Using the semiconjugacy $f\circ\pi=\pi\circ F$ and the definition of $L$ we have
\begin{eqnarray*} 
\pi_{*}\tilde{\mu}=\lim_{k\to\infty}\frac{1}{n_k}\sum_{j=0}^{n_k-1}\pi_{*}(F^{j})_{*}\nu&=&\lim_{k\to\infty}\frac{1}{n_k}\sum_{j=0}^{n_k-1}(f^{j})_{*}\pi_{*}\nu\\
&=&\lim_{k\to\infty}\frac{1}{n_k}\sum_{j=0}^{n_k-1}(f^{j}\circ\pi)_{*}\nu\\
&=&\lim_{k\to\infty}\frac{1}{n_k}\sum_{j=0}^{n_k-1}(f^{j})_{*}{\mu}={\mu}.\end{eqnarray*}
The uniqueness of $\tilde{\mu}$ is a consequence of the contraction on the fibers. In fact, suppose that $\tilde{\mu}_1$ and $\tilde{\mu}_2$ are ergodic measures which satisfy $$\pi_{*}\tilde{\mu}_1=\mu=\pi_{\ast}\tilde{\mu}_2.$$ 
Let $B_{\tilde{\mu}_1}(F)$ and $B_{\tilde{\mu}_2}(F)$ be the ergodic basins of $\tilde{\mu}_1$ and $\tilde{\mu}_2$, respectively. From the uniform contraction on fibers and the ergodicity of $\tilde{\mu}_1,\tilde{\mu}_2$, there are Borel sets $A_1,A_2\subset M$ with $A_{1}\cap A_{2}=\emptyset$ such that
$$B_{\tilde{\mu}_1}(F)=A_{1}\times N\quad\text{and}\quad B_{\tilde{\mu}_2}(F)=A_{2}\times N .$$  
On the other hand, since $\mu$ is ergodic and the sets $\pi(B_{\tilde{\mu}_1}(F))$, $\pi(B_{\tilde{\mu}_2}(F))$ are $f$-invariant, it follows that 
$$\mu\big(\pi(B_{\tilde{\mu}_1}(F))\big)=\mu\big(\pi(B_{\tilde{\mu}_2}(F)\big)\big)=1.$$
Thus we have 
 $$\pi\big(B_{\tilde{\mu}_1}(F)\big)\cap \pi\big(B_{\tilde{\mu}_2}(F)\big)\neq\emptyset$$
and by definition of $\pi$ we obtain that $A_{1}\cap A_{2}\neq\emptyset.$ This gives $\tilde{\mu}_1=\tilde{\mu}_2$.

Now we prove the second part of the lemma. First observe that the uniform contraction of $g$ on the fibers   gives   $h_{\top}(F, \pi^{-1}({x}))=0$ for every $x\in M.$ 
 Then, by Ledrappier-Walters Formula we obtain
 \begin{equation}\label{eq.tropias}
 h_{\tilde\mu}(F)=h_\mu(f)
 \end{equation}
  for any $\mu\in\mathbb P_f(M)$ and $\tilde\mu\in\mathbb P_F(M\times N)$ such that $\pi_*\tilde\mu=\mu$.
It follows from~\eqref{pressure2} and \eqref{eq.pressoes} that
\begin{eqnarray*}
P_{f}(\varphi)&\leq&P_{F}(\bar{\phi})\\
&=&\ds\sup_{\tilde{\eta}}\left\{h_{\tilde{\eta}}(F)+\int{\bar{\phi}} \,d\tilde{\eta}\right\}\\
&\leq&\!\!\!\!\sup_{\eta;\pi_{\ast}\tilde{\eta}=\eta}\left\{h_{\eta}(f)+\int{\!h_{\top}(F, \pi^{-1}(x))\,d\eta(x)}+\int{\varphi}\, d\eta\right\}\\\\
&\leq&P_{f}(\varphi).
\end{eqnarray*}
Thus we must have 
\begin{eqnarray}\label{pressas}
P_{f}(\varphi)=P_{F}(\bar{\phi}).
\end{eqnarray}
Assume now that   $\mu$ is an ergodic equilibrium state for $(f, \varphi)$. Using~\eqref{pressure2}, \eqref{eq.phis}, \eqref{eq.tropias}, \eqref{pressas}  and the fact that $\pi_*\tilde\mu=\mu$, we obtain 
$$P_{F}(\bar\phi)= P_{f}(\varphi)=  h_{ {\mu}}(f)+\int{\varphi} d \mu= h_{{\tilde\mu}}(F)+\int{\bar\phi}d\tilde\mu.$$
This shows that
 $\tilde{\mu}$    is an equilibrium state for $(F, \bar{\phi})$, and so   an equilibrium state for $(F, {\phi})$, by Proposition~\ref{homologo}.

Finally, assume that $\tilde{\mu}$    is an equilibrium state for $(F, {\phi}).$ Using again~\eqref{pressure2}, \eqref{eq.phis}, \eqref{eq.tropias}, \eqref{pressas}  and the fact that $\pi_*\tilde\mu=\mu$, we obtain 
$$P_{f}(\varphi)=P_{F}(\bar\phi)=  h_{{\tilde\mu}}(F)+\int{\bar\phi}d\tilde\mu= h_{ {\mu}}(f)+\int{\varphi} d \mu ,$$ 
thus finishing the proof of the result.
\end{proof}

\begin{Lema}\label{le.fizinho}
 If $\phi$ is a $c$-hyperbolic potential for $F$, then $\varphi$ is  a $c$-hyperbolic potential for $f$.
 \end{Lema}
\begin{proof} First observe that by Proposition~\ref{homologo} we have that $\bar\phi=\varphi \circ \pi$ is also a $c$-hyperbolic potential for $F$.  As  $F$ contracts on fibers, for each $\delta >0$ and each $(x,z) \in M\times N$ we have
$$  B^f_{\delta}(x,n) = \pi (B^F_\delta((x,z),n) ),$$
where the superscripts indicate the dynamics with respect to which we take the dynamic balls.
Recalling the definition of relative topological pressure given in Section~\ref{hyp pot}, for all  $\delta >0$, $n \in \mathbb{N}$ and $(x,z) \in M\times N$ we have
\begin{eqnarray*}
R_{n,\delta}\bar\phi(x,z)&=&   \sup_{(y,w)\in B^F_{\delta}((x,z),n)}S_n \bar\phi(y,w) \\ &=& \sup_{(y,w)\in B^F_{\delta}((x,z),n)}S_n (\varphi\circ\pi)(y,w) \\
&=& \sup_{y\in \pi (B^F_{\delta}((x,z),n))}S_n \varphi (y) \\ & = & \sup_{y \in B^f_{\delta}(x,n)} S_n \varphi(y) \\
&=&R_{n,\delta} \varphi(x) .
\end{eqnarray*}
Hence,  for each $N\in \mathbb{N}$ and $\gamma>0$ we have
$$  m_f(\varphi, \left(\Sigma_c(f)\right)^c, \delta, N, j) =    m_F(\bar\phi, \left(\Sigma_c(f)\right)^c \times N, \delta, N, j)$$ 
and
$$  m_f(\varphi, \left(\Sigma_c(f)\right), \delta, N, j) =    m_F(\bar\phi, \left(\Sigma_c(f)\right) \times N, \delta, N, j).$$ 
This yields  $$P_{f}(\varphi, \left(\Sigma_c(f)\right)^c)= P_{F}(\bar{\phi}, \left(\Sigma_c(f)\right)^{c}\times N)$$
and  $$P_{f}(\varphi, \Sigma_c(f))=P_{F}(\bar{\phi}, \Sigma_c(f)\times N).$$
Therefore  
$$P_{f}(\varphi, \left(\Sigma_c(f)\right)^c)
= P_{F}(\bar{\phi}, \left(\Sigma_c(f)\right)^{c}\times N) < P_F(\bar{\phi}, \Sigma_c(f)\times N) = P_f(\varphi, \Sigma_c(f)) ,$$
and    using \eqref{pressas}, we obtain
$$ P_{f}(\varphi, \left(\Sigma_c(f)\right)^c)<P_f(\varphi, \Sigma_c(f))=P_F(\bar{\phi}, \Sigma_c(f)\times N)=P_{F}(\bar{\phi})=P_{f}(\varphi).$$
This shows that $\varphi$ is a $c$-hyperbolic potential. 
\end{proof}
The existence and finiteness of equilibrium states for the base dynamics~$f$  was obtained in \cite{RV16}. Combining this fact with Lemma~\ref{bijection} we conclude that there exist finitely many ergodic equilibrium states for $(F, \bar{\phi}).$ Since $\bar{\phi}$ is homologous to the initial potential $\phi$ we have completed the proof of Theorem~\ref{unicidade}.

\smallskip

Now we are able to prove Theorem~\ref{statistical}. Given $(F_0,\phi_0)\in\mathcal G_c$, consider a sequence $(F_n,\phi_n)_n$  in $\mathcal G_c$  converging to $(F_0, \phi_0)$ in the product topology. Let  $\tilde{\mu}_n$ be an equilibrium state of  $(F_n, \phi_n)$. 
We are going to  show that every accumulation point $\tilde{\mu}$ of the sequence $(\tilde{\mu}_n)_n$
is an equilibrium state for  $(F_0,\phi_0).$ 

For each $n\ge0$ let $\bar\phi_n$ be the potential associated  to $\phi_n$ by Proposition~\ref{homologo}, and let $\varphi_n$ be the potential on $M$ induced by $\bar\phi_n$. 
 It follows from 
 Lemma~\ref{le.fizinho} that $\varphi_n$ is $c$-hyperbolic (with respect to $f_n$), and this means that $(f_n, \varphi_n) \in \mathcal{H}_c$  for all $n\ge0$. Moreover, using~\eqref{eq.barraphi} and~\eqref{eq.fizinho}, we have that
   the convergence of $(F_n, \phi_n)$ to $(F_0, \phi_0)$ in the product topology implies the convergence of $(f_n, \varphi_n)$ to $(f_0, \varphi_0) $.


For each $n \!\in\! \mathbb{N}$ consider ${\mu}_n =\pi_{\ast}\tilde{\mu}_n$. From Lemma~\ref{bijection} we have that $\mu_n$ is an equilibrium state  for $(f_n, \varphi_n)$. Since the projection $\pi$ is continuous, if $\tilde{\mu}$ is an accumulation point of the sequence $(\tilde{\mu}_n)_n$, then $\mu=\pi_{\ast}\tilde{\mu}$ is an accumulation point of $(\mu_n)_n$. By the equilibrium stability   given by Theorem~\ref{statistical para f}, we have that $\mu$ is an equilibrium state for $(f_0, \varphi_0)$. Hence, applying Lemma~\ref{bijection} again we have that $\tilde{\mu}$ is an equilibrium state for  $(F_0,\phi_0)$. This concludes the proof of Theorem~\ref{statistical}.

We finish this section deducing the continuity of the topological pressure on the set $\mathcal{G}_c$. 

\begin{Cor} \label{continuidade da pressao}
The   function $$
\begin{array}{cccc}
\mathcal G_c  \ & \longrightarrow  \ & \mathbb{R} \\
    (F,\phi) \ & \longmapsto  \ & P_F(\phi)
\end{array}
$$
is continuous.
\end{Cor}

\begin{proof} As before, given $(F, \phi)\in \mathcal G_c$ consider the induced $(f, \varphi)\in \mathcal{H}_c$. Using~\eqref{eq.barraphi} and~\eqref{eq.fizinho} it follows  that if $(F,\phi) $ varies continuously in $\mathcal G_c$, then   $(f, \varphi)$ also varies continuously in $ \mathcal{H}_c$. Considering $\bar\phi$ given by Proposition~\ref{homologo}, we have $P_{f}(\phi)=P_{F}(\bar{\phi})$. It follows from~\eqref{pressas} that $P_{f}(\varphi)=P_{F}(\phi)$.
%
%
Hence, as by Corollary~\ref{continuidade da pressao na base} we have   $P_{f}(\varphi)$ varying continuously within $\mathcal{H}_c$, then $P_{F}(\phi)$ varies continuously within $\mathcal G_c$ as well.
\end{proof}


\section{Applications}\label{applications}

In this section we present some classes of systems which satisfy our results. We begin describing a robust class of non-uniformly expanding maps studied by Alves, Bonatti and Viana \cite{ABV}, Arbieto, Matheus and Oliveira \cite{AMO}, Oliveira and Viana \cite{OV08}, Varandas and Viana \cite{VV}. 

\begin{Ex} \normalfont{Let $M$ be a compact manifold and let $f:M\rightarrow M$ be a $C^{1}$ local diffeomorphism. Fixing   $\delta>0$ small and $\sigma<1$, consider  a covering $\mathcal P=\left\{P_{1},..., P_{q}, P_{q+1},..., P_{s}\right\}$ of $M$ by domains of injectivity for $f$ and a  region $\textsl{A}\subset M$ satisfying:
\begin{enumerate}
\item[(H1)]   $\|Df^{-1}(x)\|\leq 1+\delta$, for every $x\in\textsl{A}$;
\item[(H2)] $\|Df^{-1}(x)\|\leq \sigma $, for every $x\in M\setminus\textsl{A}$;
\item[(H3)]   $A$   can be covered by $q$ elements of the partition $\mathcal P$ with $q<\deg(f).$
\end{enumerate}
The authors aforementioned showed that there exists a  constant $c>0$ and a set $H\subset M$ such that for every $x\in H$ we have
\begin{eqnarray*} 
\limsup_{n\rightarrow+\infty}\frac{1}{n}\sum_{i=0}^{n-1}\log\|Df(f^{j}(x))^{-1}\|\leq-c .
\end{eqnarray*}
Moreover, if $\phi:M\rightarrow\mathbb{R}$ is a H\"older continuous potential with \emph{small variation}, i.e. $$\sup\phi - \inf\phi < \log\deg(f)-\log q,$$ then the relative pressure $P(\phi, H)$ satisfies 
$$P_{f}(\phi, H^{c})<P_{f}(\phi, H)=P_{f}(\phi).$$
and thus $\phi$ is $c$-hyperbolic for $f$. 

Let $\mathcal{F}$ be the class of $C^1$ local diffeomorphisms satisfying the conditions (H1)-(H3) and assume that for every $f\in \mathcal{F}$ and $x\in M$  the set $\{f^{-n}(x)\}_{n\geq 0}$ is dense in $M$. Consider the family
 $$\mathcal H=\{(f,\phi): f\in \mathcal{F} \text{ and $\phi:M\to\mathbb R$ H\"older continuous with small variation}  \}.$$
Since the constant $c>0$ is uniform in the class $\mathcal F$ we can apply Theorem~\ref{statistical para f} to conclude that the family $\mathcal H$ is equilibrium stable.}
\end{Ex}


An interesting and more specific case of application of Theorem~\ref{statistical para f} is the family of intermittent maps described in the next example. This can be seen as a particular case of the previous example. 

\begin{Ex}\normalfont{Fix some positive constant $\alpha\in(0,1)$ and define on $S^{1}$ the local difeomorphism 
$$
f_{\alpha}(x)=\left\{
\begin{array}{cc}
x(1+2^{\alpha}x^{\alpha}),&\; \textit{\normalfont{if}}\;\; 0\leq x\leq \frac{1}{2}\\\\
x-2^{\alpha}(1-x)^{1+\alpha},&\; \textit{\normalfont{if}}\;\; \frac{1}{2}\leq x\leq 1.
\end{array}
\right.
$$
Since $f_{\alpha}$ is strongly mixing (for each open set $U$ there exists $N\in\mathbb{N}$ such that $f_{\alpha}^{N}(U) = M$), then for every $f_\alpha$ and $x\in S^1$  the set $\{f_\alpha^{-n}(x)\}_{n\geq 0}$ is dense in $S^1$. Let $\mathcal{F}$ be the class of $C^1$ local diffeomorphisms $\{f_{\alpha}\}_{\alpha\in(0,1)}$ and consider the family
 $$\mathcal H=\{(f,\phi): f\in \mathcal{F} \text{ and $\phi:S^1\to\mathbb R$ H\"older continuous with small variation}   \}.$$
Hence, the family $\mathcal H$ satisfies the hypotheses of Theorem~\ref{statistical para f} and therefore $\mathcal H$ is equilibrium stable.}
\end{Ex}


Now we present a class of maps and potentials for which the techniques used to prove Theorem~\ref{unicidade} and Theorem~\ref{statistical} can be applied. This is a family of partially hyperbolic horseshoes whose dynamics is given by a step skew product over a horseshoe. This class of maps was defined in \cite{diazetal} and has been studied in the works \cite {LOR},  \cite{RS15} and \cite{RS16}. 
Here 
we obtain the equilibrium stability of this family. 

\begin{Ex} \label{ferradura} \normalfont{Consider the cube $R= [0,1]\times[0,1]\times[0,1]\subset\mathbb{R}^3$  and   the paralellepipeds
$$R_0 =[0,1]\times [0,1]\times [0,1/6]\quad \mbox{and} \quad R _1=[0,1]\times [0,1]\times [5/6,1].$$ 
Consider a map   defined for $(x,y,z)\in R_0$ as  
   $$ F_{0}(x,y,z) =(\rho x , f(y),\beta z),$$
where $0 < \rho <{1/3}$, $\beta> 6$ and  $$f(y) =\frac {1}{1 - \left(1-\frac{1}{y}\right)e^{-1}}.$$
Consider also a map  defined for $(x,y,z)\in R_1$ as
$$F_{1}(x,y,z)  = \left(\frac{3}{4}- \rho x , \sigma (1 - y) ,\beta_{1} \left(z - \frac{5}{6} \right)\right),$$
where  $0<\sigma< {1/3}$ and $3< \beta_1 < 4$.
We define the horseshoe  map $F$ on $R$ as
 $$F\vert_{R_0}=F_0,\quad F\vert_{R_1}=F_1,
 $$
with  $  R\setminus(R_0\cup R_1)$ being mapped injectively outside $R$.

In \cite{diazetal} it was proved that the non-wandering set of $F$ is partially hyperbolic when we consider fixed parameters satisfying conditions above.  
In \cite{RS15} it was considered the map $F^{-1}$ and proved that it can be written as a skew product whose base dynamics, denoted by $G$, is strongly topologically mixing and non-uniformly expanding. Moreover the fiber dynamics is a uniform contraction. Consider $\Omega$ the maximal invariant set for $F^{-1}$  on the cube $R$.

We say that a H\"older continuous potential  $\phi:R_0\cup R_1\to \mathbb{R}$ has \emph{small variation} if
\begin{equation}\label{potencial ferradura}
\sup\phi - \inf\phi < \frac{\log\omega}{2}, \quad \mbox{where}\quad \displaystyle\omega=\frac{1+\sqrt{5}}{2}.
\end{equation}
Let $\mathcal{S}$ be the family of horseshoes described above and note that it depends on the parameters $\rho, \beta, \beta_1$ and $\sigma$. Considering the family
$$
\tilde{\mathcal{G}}=\left\{(F^{-1}, \phi) :  F\in \mathcal{S}\text{ and $ \phi:\Omega \to\mathbb R$ H\"older continuous satisfying \eqref{potencial ferradura}} \right\} .
$$
We are going to verify that $\tilde{\mathcal{G}}$ is equilibrium stable. Let $(F^{-1}_n,\phi_n)_n$ be a sequence in $\tilde{\mathcal{G}}$ converging to $(F^{-1}_0, \phi_0)$ in the  product topology. By Lemma~\ref{le.fizinho}, for each $n\geq 0$ the potential $\phi_n$ induces a potential $\varphi_n$ which is hyperbolic with respect to the base dynamics $G_n.$ Moreover the convergence of $(F^{-1}_n, \phi_n)$ to $(F^{-1}_0, \phi_0) $ implies  the convergence of $(G_n, \varphi_n)$ to $(G_0, \varphi_0)$.

Consider $\tilde{\mu}_n$ the equilibrium state for  $(F^{-1}_n, \phi_n)$. From Lemma~\ref{bijection} we know that the push-foward ${\mu}_n =\pi_{\ast}\tilde{\mu}_n$ is an equilibrium state for $(G_n, \varphi_n)$ for every $n\ge 1$. Hence if $\tilde{\mu}$ is an accumulation point of the sequence $(\tilde{\mu}_n)_n$ then $\mu=\pi_{\ast}\tilde{\mu}$ is an accumulation point of $(\mu_n)_n$. Since the base dynamics $G_n$ is strongly mixing and non-uniformly expanding and the potential $\varphi_n$ is hyperbolic, it follows from Theorem~\ref{statistical para f} that $\mu$ is an equilibrium state for $(G_0, \varphi_0)$. Hence, applying Lemma~\ref{bijection} again we have that $\tilde{\mu}$ is an equilibrium state for  $(F^{-1}_0,\phi_0)$. 

Now define the family
$$\mathcal{G}=\left\{(F, \phi) :   F\in \mathcal{S}\text{ and $ \phi:\Omega \to\mathbb R$ H\"older continuous satisfying \eqref{potencial ferradura}} \right\}. $$
Since each $F \in \mathcal{S}$ is a diffeomorphism, the set of equilibrium states for $(F, \phi) \in \mathcal{G} $  coincides with the one for $(F^{-1}, \phi) \in\tilde{\mathcal{G}} $. Thus $\mathcal{G}$ is equilibrium stable as well. }
\end{Ex}



\begin{thebibliography}{99}

\bibitem{Alves} Alves, J. F., {\em SRB measures for non-hyperbolic systems with multidimensional expansion}, Ann. Sci. Ec. Norm. Super. 33 (2000) pp 1-32. 

\bibitem{ABV} Alves, J. F., Bonatti, C. , Viana, M., {\em SRB measures for partially hyperbolic systems whose central direction is mostly expanding}, Inventiones Math, 140 (2000), pp 351-398.

\bibitem{AV} Alves, J. F., Viana, M., {\em Statistical stability for robust classes of maps
with non-uniform expansion},  Ergodic Theory Dynam. Systems, 22 (2002), pp 1-32. 

\bibitem{Ar} Ara\'ujo, V., {\em Semicontinuity of entropy, existence of equilibrium states and continuity of physical measures},
Discrete Contin. Dyn. Syst. 17 (2007), no. 2, pp 371?386. 

\bibitem{APo} Andronov, A., Pontryagin, L., {\em  Syst\`{e}mes grossiers}, Dokl. Akad. Nauk. USSR 14 (1937), pp 247-251.

\bibitem{AMO} Arbieto, A., Matheus C., Oliveira, K., \textit{ Equilibrium states for random non-uniformly expanding maps}. Nonlinearity, 17 (2004), pp 581-593.

\bibitem{BCV} Bomfim, T.; Castro, A.; Varandas, P. {\em Differentiability of thermodynamical quantities in non-uniformly expanding dynamics}, Adv. Math. 292 (2016), pp 478-528.

\bibitem{Bowen71} Bowen, R., {\em Entropy for Group Endomorphisms and Homogeneous Spaces}, Trans. Amer. Math. Soc., 153  (1971), pp 401-414.

\bibitem{Buzziecia} Buzzi, J., Fisher, T., Sambarino, M., V\'asquez, C., {\em Maximal entropy measures for
certain partially hyperbolic, derived from Anosov systems},  Ergodic Theory Dynam. Systems, 32 (2011), pp 63-79.

\bibitem{Buzzi} Buzzi, J., Pacault, F., Schmitt, B., {\em Conformal measures for multidimensional piecewise invertible maps},  Erg. Theory and Dyn. Systems, 21 (2001), pp 1035-1049.

\bibitem{BuzziSarig} Buzzi, J., Sarig, O., {\em Uniqueness of equilibrium measures for countable Markov shifts}, Erg. Theory and Dyn. Systems, 23 (2003), pp 1383-1400.

\bibitem{CN} Castro, A., Nascimento, T., {\em Statistical properties of the maximal entropy measure for partially hyperbolic attractors}, Erg. Theory and Dyn. Systems 37 (2017) pp 1060-1101. 

\bibitem{CV} Castro, A., Varandas, P., {\em Equilibrium states for non-uniformly expanding maps: decay of correlations and strong stability} Ann. Inst. H. Poincar\'e Anal. Non Lin\'eaire 30 (2013), no. 2, pp 225-249. 

\bibitem{CFT16} Climenhaga, Vaughn, Fisher, Todd, Thompson, Daniel J., {\em Unique equilibrium states for Bonatti-Viana diffeomorphisms}, preprint 2016.

\bibitem{CFT17} Climenhaga, Vaughn, Fisher, Todd, Thompson, Daniel J., {\em Equilibrium states for Ma\~n\'e diffeomorphisms}, preprint 2017.



\bibitem{Deimling} Deimling, K., {\em Nonlinear Functional Analysis}, Springer, Berlin, 1985.

\bibitem{diazetal} D\'\i az, L.J., Horita, V., Rios, I., Sambarino, M., {\em Destroying horseshoes via heterodimensional cycles: generating bifurcations inside homoclinic classes},  Ergodic Theory Dynam. Systems, 29 (2009), pp 433-474.

\bibitem{Led} Ledrappier, F.,  Walters, P., {\em A relativised variational principle for continuous  transformations}, J. London Math. Soc. (2) 16 (1977), pp 568-576.

\bibitem{LOR} Leplaideur, R., Oliveira, K., Rios, I., {\em Equilibrium States for partially hyperbolic horseshoes},  Ergodic Theory Dynam. Systems, 31 (2011), pp 179-195.

\bibitem{OV08} Oliveira, K., Viana M., {\em Thermodynamical formalism for robust classes of potentials and non-uniformly hyperbolic maps},  Erg. Th. Dyn. Systems, 28 (2008), pp 501-533.


\bibitem{Pes} Pesin, Y. {\em Dimension theory in dynamical systems: Contemporary views and applications}, University of Chicago Press (1997).

\bibitem{RS16} Ramos, V., Siqueira, J., {\em On equilibrium states for partially hyperbolic horseshoes: uniqueness and statistical properties}, Bulletin of Brazillian Mathematical Society, 48 (2017), pp 347-375.

\bibitem{RV16} Ramos, V., Viana, M., {\em Equilibrium States for Hyperbolic Potentials}, Nonlinearity, 30 (2017), pp 825-847.

\bibitem{RS15} Rios, I., Siqueira, J., {\em On equilibrium states for partially hyperbolic horseshoes},  Ergodic Theory Dynam. Systems, 21 (2016), pp 1-35.

\bibitem{Ruelle68} Ruelle, D.,  {\em Statistical Mechanics of a One-dimensional Lattice Gas}, Communications in Mathematical Physics, 9 no. 4 (1968), pp 267-278.

\bibitem{Ruelle78} Ruelle, D., {\em Thermodynamic formalism}, Encyclopedia Mathematics and its Applications, vol 5, Addison-Wesley Publishing Company, 1978.

\bibitem{Sarig99} Sarig, O., {\em Thermodynamic formalism for countable Markov shifts},  Ergodic Theory Dynam. Systems, 19 (1999), pp 1565-1593.

\bibitem{Sarig03} Sarig, O., {\em Existence of Gibbs measures for countable Markov shifts}, Proc. Amer. Math. Soc., 131 (2003), pp 1751-1758 (Electronic).

\bibitem{Sinai} Sinai, Y., {\em Gibbs measures in ergodic theory}, Russ. Math. Surveys, 27 (1972), pp 21-69.

\bibitem{VV} Varandas, P., Viana, M., {\em Existence, uniqueness and stability of equilibrium states for non-hyperbolic expanding maps}, Ann. I. H. Poincar\'e, AN 27 (2010), pp 555-593.

\bibitem{VO} Viana, M.,  Oliveira, K, {\em Foundations of Ergodic Theory}, Cambridge Studies in Advanced Mathematics, 151. Cambridge University Press, Cambridge, 2016. 

\bibitem{W} Walters, P., {\em A variational principle for the pressure of continuous transformations}, Amer.  J. Math. 97 (1975), pp 937-97.


\end{thebibliography}
\end{document}